\documentclass[12pt]{article}

\usepackage{pstricks}
\usepackage{graphicx,psfrag}
\usepackage{amsmath}
\usepackage{amssymb}
\usepackage{amsthm}

      \newcommand {\al}   {\alpha}          \newcommand {\bt}  {\beta}

                 \newcommand {\vphi} {\varphi}
      \newcommand {\lam}  {\lambda}

      \newcommand {\pl}   {\partial}

      \newcommand {\RRR}  {{\mathbb R}}

     \newcommand {\beq}  {\begin{equation}}
      \newcommand {\eeq}  {\end{equation}}

      \newtheorem{theorem}{Theorem}
      \newtheorem{lemma}{Lemma}

      \newtheorem{predl}{Proposition}
      \newtheorem{primer}{Example}
      \newtheorem{corollary}{Corollary}
      
     \newtheorem{question}{Question}

\title{Billiard transformations of parallel flows: a periscope theorem}

\author{Alexander Plakhov\thanks{Center for R\&{}D in Mathematics and Applications, Department of Mathematics, University of Aveiro, Portugal, and Institute for Information Transmission Problems, Moscow, Russia.} \and Sergei Tabachnikov\thanks{Department of Mathematics, Penn  State University.} \and  Dmitry Treschev\thanks{Steklov Mathematical Institute of Russian Academy of Science and Lomonosov Moscow State University.}}
\date{}

\begin{document}

\maketitle

\begin{abstract}
We consider the following problem: given two parallel and identically oriented bundles of light rays in $\RRR^{n+1}$ and given a diffeomorphism between the rays of the former bundle and the rays of the latter one, is it possible to realize this diffeomorphism by means of several mirror reflections? We prove that a 2-mirror realization is possible if and only if the diffeomorphism is the gradient of a function. We further prove that any orientation reversing diffeomorphism of domains in $\RRR^2$ is locally the composition of two gradient diffeomorphisms, and therefore can be realized by 4 mirror reflections of light rays in $\RRR^3$, while an orientation preserving diffeomorphism can be realized by 6 reflections. In general, we prove that an (orientation reversing or preserving) diffeomorphism of wave fronts of two normal families of light rays in $\RRR^3$ can be realized by 6 or 7 reflections.

\end{abstract}

\begin{quote}
{\small {\bf Mathematics subject classifications:} Primary 49Q10, 49K30}
\end{quote}

\begin{quote}
{\small {\bf Key words and phrases:} billiards, freeform surfaces, geometrical optics, imaging
}
\end{quote}

\section{Introduction}

This paper concerns geometrical optics, a classical subject that goes all the way back to Fermat, Huygens,  Newton, and that remains a research area of a considerable contemporary interest.

One of the reasons for this interest is that, fairly recently, industrial methods were developed for manufacturing freeform reflective surfaces of optical quality.

From the mathematical point of view, a freeform mirror is a smooth hypersurface in Euclidean space from which the rays of light reflect according to the familiar law ``the angle of incidence equals the angle of reflection". We refer to \cite{Koz} for a historical introduction to geometrical optic and to \cite{Bo} for an encyclopedic study of the subject. We cannot help mentioning a classical source, the treatise by W.R. Hamilton \cite{Ha}.

Another reason for the popularity of geometrical optics is its close relation with the ever-growing study of mathematical billiards; the reader interested  in billiards is referred to \cite{K-T,Pl,Ta}. Mathematical billiards describe the motion of a mass-point inside a domain with the elastic reflection off the boundary given by the  law of equal angles.

Still another reason for which geometrical optics continues to attract attention is that the space of oriented lines, i.e., rays of light, in $\RRR^{n+1}$ is an example of a symplectic manifold (symplectomorphic to the cotangent bundle $T^* S^n$). The optical, or billiard, reflection in a mirror defines a  symplectic transformation of the space of lines; see \cite{Ar,Ta} for a modern treatment.

Only a negligible part of symplectic transformations of the space of lines is realized by the composition of reflections in mirrors. Indeed, a symplectic transformation of the space of rays in $\RRR^{n+1}$ is given by its generating function, a function of $2n$ variables, whereas a mirror is locally the graph of a function of only $n$ variables. It is an interesting and, to the best of our knowledge,  completely open  problem to characterize the symplectic transformations of the space of oriented lines that arise as consecutive mirror reflections.

A common object of study in geometrical optics is a normal family of lines in $\RRR^{n+1}$, that is,  an $n$-parameter family of oriented lines perpendicular to a hypersurface (think of this surface as emanating light).
The hypersurface is called a wave front. A normal family has a one-parameter family of wave fronts; they are equidistant from each other.

From the symplectic point of view, the normal families are Lagrangian submanifolds in the symplectic space of rays. Since an  optical reflection is a symplectic transformation, normal  families are transformed to normal families. This is the classic Malus theorem; see \cite{Ma} for a modern account.

Conversely, given two generic local normal families consisting of the outgoing and the incoming rays, there is a one-parameter family of mirrors that reflect one family to the other. This is Levi-Civita's theorem  \cite{LC}. The mirrors are the loci of points for which  the sum of distances to the respective wave fronts is constant.

For example, in dimension two, if the two normal families consist of lines thought points $A$ and $B$, then the respective mirrors are the ellipses with the foci $A$ and $B$. Identify the circles centered at $A$ and $B$ with the projective line via stereographic projections. Then  the respective  mappings of the normal families are M\"obius transformations, see Appendix and \cite{Fr}.

In general, it is an interesting open problem to describe the one-parameter family of mappings of  normal families given by a one-mirror reflection. Another problem, motivated by applications, is as follows.

Consider a diffeomorphism of two normal families of rays in $\RRR^{n+1}$. We wish to realize this diffeomorphism as the composition of a number of mirror reflections. In particular, what is the least number of mirrors needed?

In this paper we consider a particular case of this problem: the two normal families consist of two parallel and identically oriented bundles of rays in $\RRR^{n+1}$. We think of a system of mirrors that takes a parallel beam to a parallel beam as a periscope.

In Section \ref{sec:2-mirror}, we show that the diffeomorphism of parallel beams realized by a two-mirror reflection is a gradient diffeomorphism and, conversely,  if the diffeomorphism is gradient, it can be realized by two mirrors. We describe the mirrors explicitly: they form a 2-parameter family determined by the diffeomorphism.

In Section \ref{sec:mirr_trans}, we consider the case $n=2$. Then we have a diffeomorphism between two compact domains in $\RRR^2$. We show that if the diffeomorphism is orientation reversing then it is a composition of two gradient diffeomorphisms, and hence it can be realized by a four-mirror reflection. We present an example of an orientation preserving diffeomorphism that is not the composition of two gradient diffeomorphisms. We show that orientation preserving diffeomorphisms can be realized by six-mirror reflections. As a consequence, a diffeomorphism between two normal families of rays in $\RRR^3$ can be realized by an at most seven-mirror reflection.

In the last Section \ref{problems}, we present a collection of open problems on realization of diffeomorphisms by mirror reflections.

The literature on freeform mirrors is substantial, and we mention here but a few  relevant papers.

The papers \cite{O,GO} concern two-reflector systems that transform an incoming planar wave front in Euclidean space into outgoing planar wave front with a prescribed output intensity. This problem is formulated and solved as a mass transfer problem. The observation that a two-mirror reflection defines a gradient diffeomorphism, which is a part of our Theorem \ref{t1}, is made in these papers: equations (4.43) in \cite{O} and (2.3) in \cite{GO}.

The paper \cite{CH} concerns mirror realizations of a mapping of a 2-parameter family of rays in $\RRR^3$ to another such family; the families are not necessarily normal. Using the Cartan-K\"ahler theorem in the theory of exterior differential systems, a numerical method is described for constructing four mirrors that realize the mapping (the presented examples involve only normal families of rays).

The recent paper \cite{Pe} concerns a version of the question which symplectic transformations of the space of oriented lines are realized by consecutive mirror reflections. Consider a double-mirror system consisting of two infinitesimally close mirrors (thin film). A ray of light goes through the first mirror, reflects in the second one, then reflects in the first one, and escapes by going through the second mirror. This defines an infinitesimal symplectic transformation of the space of rays, that is, a Hamiltonian vector field. These Hamiltonian vector fields are described in \cite{Pe} in terms of the geometry of the thin film.

Another application of mirror transformations of normal families is related to the phenomenon of invisibility, where the light rays go round a certain domain in Euclidean space (called an {\it invisible body}), while the corresponding transformation of a wave front is the identity (that is, the light rays are not preserved as a result of mirror reflections). A review of results on billiard invisibility can be found in chapter 8 of \cite{Pl}; see also the recent paper \cite{Pl_arX} where a 2D body invisible for arbitrarily many parallel flows is constructed.

Finally, let us mention another problem of mirror design presented in \cite{Tre_PhysD,Tre_Proc}, where the existence of billiard tables with locally linearizable dynamics is studied.

\section{2-mirror transformations}
\label{sec:2-mirror}

We are concerned with the following question: given two parallel bundles of light rays in the same direction and given a (smooth) one-to-one correspondence between the rays of the former bundle and the rays of the latter one, is it possible to realize this correspondence by means of several mirror reflections? If the answer is yes, what is the minimum number of mirrors and/or mirror reflections needed?

More precisely, choose an orthonormal coordinate system $x_1, \ldots, x_n,\, z$ in Euclidean space $\RRR^{n+1}$ and consider two bounded domains $D_1$ and $D_2$ in $\RRR^n$ and a diffeomorphism $f : D_1 \to D_2$. Consider two bundles of rays codirectional with the $z$-axis, with the cross sections $D_1$ and $D_2$. The light rays of the former bundle are naturally labeled by $x = (x_1, \ldots, x_n) \in D_1$, and the rays of the latter one by $x \in D_2$. The problem is to find a finite collection of hypersurfaces $\{ S_i \}$ so that, for all $x \in D_1$, the light ray of the former bundle labeled by $x$, after several reflections from the surfaces, is transformed into the ray of the latter bundle labeled by $f(x)$.

In terms of billiard dynamics, we are dealing with the billiard in $\RRR^{n+1} \setminus (\cup_i S_i)$. The desirable dynamics is the following. A particle from the former bundle, whose position for small values of time $t$ is $(x,t)$,\, $x \in D_1$, makes several consecutive reflections off the surfaces $S_i$ and finally becomes a part of the latter bundle and has the position $(f(x), t - \tau(x))$ for $t$ sufficiently large.

Notice that the time shift $\tau(x)$ is the same for all particles, $\tau(x) = c$, and depends only on the choice of the mirrors. This is a well known fact of geometrical optics: ``the optical path length between two wave fronts is the same for all rays", see, e.g., \cite{Bo}, Section 3.3.3 and Appendix.

Consider a couple of simple examples. First, the translation by a vector $a \in \RRR^n$,\, $f : x \mapsto x + a$ (where $D_2 = D_1 + a$ and $D_1 \cap D_2 = \emptyset$) can be realized by two plane (and parallel) mirrors. Second, a dilation $f : x \mapsto kx$ (where $k \ne 0,\, 1$ and $D_2 = k D_1$,\, $D_1 \cap D_2 = \emptyset$) can be realized by two mirrors which are pieces of paraboloids of rotation, $z = (|x|^2 - c^2)/(2c)$ and $z = (|x|^2 - c^2 k^2)/(2ck)$, where $c$ is an arbitrary positive parameter.

It is convenient to use the auxiliary function $g(x) = f(x) - x$ defined on $D_1$; then the diffeomorphism takes the form $x \mapsto x + g(x)$.

The first question we are addressing is: what diffeomorphisms $x \mapsto x + g(x)$ can be realized by reflections off only two mirrors? (Notice that the two corresponding mirrors should be graphs of functions, say $\Phi_1 : D_1 \to \RRR$ and $\Phi_2 : D_2 \to \RRR$.) The answer is given by the following theorem.

\begin{theorem}[Periscope Theorem]\label{t1}
(a) If $g$ is realized by a 2-mirror system, then $g = \nabla G$ for a smooth function $G : D_1 \to \RRR$. Moreover, one has $G(x) = c\Phi_1(x)$, where the function $\Phi_1$ defines the first mirror and $c$ is a positive constant.

(b) If $D_1$ and $D_2$ are convex and disjoint and $g = \nabla G$ for a smooth function $G : D_1 \to \RRR$, then $g$ can be realized by infinitely many 2-mirror systems. These systems form a two-parameter family, with the first mirror in each system being the graph of a function $\Phi_1(x) = \frac 1c\, G(x) + h$, with arbitrary $h$ and with the parameter $c > 0$ sufficiently large. The parameter $h$ determines the ``height" of the 2-mirror system.
\end{theorem}

\begin{proof}
(a) Let $g$ be realized by the mirrors graph$(\Phi_1)$ and graph$(\Phi_2)$. A particle labeled by $x \in D_1$ is reflected at the points $A_1 = (x, \Phi_1(x))$ and
\beq\label{A_2}
A_2 = (x + g(x), \Phi_2(x + g(x)))
\eeq
(see Fig.~\ref{fig_tr2}).
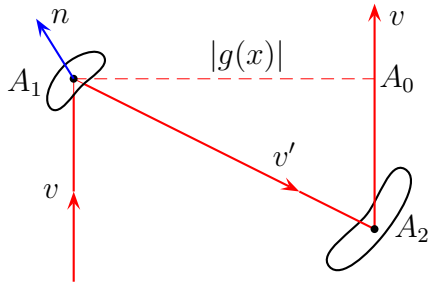
\begin{figure}[h]
\begin{picture}(0,120)
\rput(5.67,-0.1){
\scalebox{1}{
\psline[linewidth=0.3pt,linecolor=red,linestyle=dashed](0,3)(4,3)
\psecurve(4.1,0.9)(3.4,0.5)(3.9,1.2)(4.4,1.8)(4.1,0.9)(3.4,0.5)(3.9,1.2)
\psline[linewidth=0.8pt,linecolor=red,arrows=->,arrowscale=1.6](0,0.3)(0,1.5)
\psline[linewidth=0.8pt,linecolor=red](0,1.5)(0,2.8)
\psline[linewidth=0.3pt,linecolor=red](0,2.8)(0,3)(0.14,2.93)
\psline[linewidth=0.8pt,linecolor=red,arrows=->,arrowscale=1.6](0.14,2.93)(3,1.5)
\psline[linewidth=0.8pt,linecolor=red,arrows=->,arrowscale=1.6](3,1.5)(4,1)(4,4)
\psecurve(0.2,2.8)(-0.3,2.6)(-0.1,3.2)(0.4,3.3)(0.1,2.9)(-0.3,2.6)(-0.1,3.2)
\rput(-0.3,1.5){$v$}
\rput(4.3,3.8){$v$}
\rput(2.8,2){$v'$}
\psdots[dotsize=3pt](0,3)(4,1)
\rput(-0.65,2.95){$A_1$}
\rput(4.5,1){$A_2$}
\psline[linecolor=blue,arrows=->,arrowscale=1.6](0,3)(-0.5,3.8)
\rput(-0.18,3.83){$n$}
\rput(2.3,3.3){$|g(x)|$}
\rput(4.3,3){$A_0$}
}
}
\end{picture}
\caption{An individual trajectory in more detail. The points of reflection are $A_1$ and $A_2.$ We also mark the point $A_0$ on the trajectory, so that $A_0A_1A_2$ is a right triangle.}
\label{fig_tr2}
\end{figure}
The unit normals to graph$(\Phi_1)$ at $A_1$ are $\pm n/|n|$, where $n = (-\nabla\Phi_1(x),\, 1)$.
The velocities of the particle before and after the first reflection are $v = (\bar 0, 1)$ and
$$
v' = v - 2\langle v, n \rangle \frac{n}{|n|^2} = \frac{(2\nabla\Phi_1(x),\, -1 + |\nabla\Phi_1(x)|^2)}{1 + |\nabla\Phi_1(x)|^2},
$$
where $\langle \cdot\,, \cdot \rangle$ is the scalar product. Thus, the point of second reflection can be written as
\beq\label{A2}
A_2 = (x, \Phi_1(x)) + c \Big( \nabla\Phi_1(x),\ \frac{-1 + |\nabla\Phi_1(x)|^2}{2} \Big),
\eeq
where $c = c(x) > 0$ generally depends on $x$. Comparing (\ref{A_2}) and (\ref{A2}), we conclude that $g(x) = c\nabla\Phi_1(x)$.

By (\ref{A2}), the slope (the tangent of the inclination angle) of the intermediate segment $A_1A_2$ is
\beq\label{tan}
\frac{-c^2 + |g(x)|^2}{2c|g(x)|}.
\eeq

Consider the point $A_0 = (x + c\nabla\Phi_1(x),\, \Phi_1(x))$ on the trajectory of the particle (see Fig.~\ref{fig_tr2}). One easily calculates the sides of the right triangle $A_0A_1A_2$,
$$
|A_0A_1| = c|\nabla\Phi_1(x)|, \quad |A_1A_2| = c\, \frac{1 + |\nabla\Phi_1(x)|^2}{2}, \quad \pm |A_0A_2| = c\, \frac{1 - |\nabla\Phi_1(x)|^2}{2},
$$
where the sign "+" or "$-$" in the latter equation is taken according as $A_2$ lies below or above $A_0$.

It remains to notice that the ``length" of the trajectories is preserved in the flow; that is, the difference $|A_1A_2| \pm |A_0A_2| = c$ is constant.

Note also that the function $\Phi_2$ can easily be determined from equation (\ref{A2}),
\beq\label{Psi}
\Phi_2(x + g(x)) = \Phi_1(x) +\frac{|g(x)|^2 - c^2}{2c}.
\eeq

(b) Suppose that $g = \nabla G$. Let the two mirrors be the graphs of the function $\Phi_1(x) = \frac 1c\, G(x)+ h$ and of the function $\Phi_2$ determined by equation (\ref{Psi}). This equation ensures preservation of the ``length" in the family of infinite polygonal lines with the vertices at $A_1 = (x,\Phi_1(x))$ and $A_2 = (x+g(x),\Phi_2(x+g(x))$. The first and the third segments of each polygonal line are ``vertical" (parallel to $v = (\bar 0, 1)$) half-lines, and are defined, respectively, by $A_1 + \lam v,\; \lam \le 0$ and $A_2 + \lam v, \; \lam \ge 0$.

Let us show that each polygonal line is actually a billiard trajectory. Indeed, one obviously has billiard reflection at $A_1$, and the ``length" preservation condition ensures that the reflection at $A_2$ also obeys the billiard law. Now we have to check that the line has no points of intersection, other than $A_1$ and $A_2$,
with the mirrors. Indeed, the first and the third segments obviously do not have such points (since $D_1$ and $D_2$ are disjoint). It remains to show that the intermediate segment $A_1A_2$ has no interior points of intersection with graph$(\Phi_1)$ and graph$(\Phi_2)$.

The slope of the line $A_1A_2$ is given by (\ref{tan}); for $c$ sufficiently large and all $x \in D_1$, it is smaller than $-{|g(x)|}/{c} = -|\nabla\Phi_1(x)|$. Now draw the ``vertical" 2D plane through $A_1A_2$; it is parallel to the vectors $(g(x), 0)$ and $v = (\bar{0}, 1)$. The section of graph$(\Phi_1)$ by this plane is the graph of a function of one variable defined on a segment (since $D_1$ is convex), with the modulus of derivative everywhere smaller than or equal to $|\nabla\Phi_1(x)|.$ On the other hand, the segment $A_1A_2$ is the graph (in the same plane) of a linear function with the modulus of derivative greater than $\max_x |\nabla\Phi_1(x)|.$ Therefore the only point of intersection of the two graphs is $A_1$.

Repeating this argument, one concludes that the only point of intersection of the segment $A_1A_2$ with graph$(\Phi_2)$ is $A_2$.
\end{proof}

Let us mention a connection of the above computations with the Legendre transform. The next Proposition is equivalent to the result in Section 4.4.2 of \cite{O}.

Let $y$, $\Psi_1$, and $\Psi_2$ be defined by the equations
$$
  y = -x -g(x), \
  \Psi_1(x) = -x^2/2 + c^2/4 - c\Phi_1(x), \   \Psi_2(y) = -y^2/2 + c^2/4 + c\Phi_2(-y).
$$

\begin{predl}
The functions $\Psi_1$ and $\Psi_2$ are related by the Legendre transform:
$$
  y = \partial\Psi_1(x) / \partial x, \quad
  \Psi_2(y) = \langle x,y\rangle - \Psi_1(x).
$$
\end{predl}

The proof is a straightforward computation. \qed
\medskip

Let us now generalize claim (b) of Theorem \ref{t1} to the case of a {\it piecewise smooth} and piecewise gradient mapping $f$. This generalization will be used in the next section where we study systems with 4 reflections.

It is straightforward to realize such a mapping $f$ by two collections of mirrors and by a bundle of polygonal lines that have billiard reflections from these mirrors. The main difficulty, however, is to prove that the polygonal lines do not have superfluous intersections with the mirrors, and thus are {\it true} billiard trajectories. The additional assumptions described below are needed to ensure this non-intersection condition.

Let $D_1$ be the union of finitely many closed domains with disjoint interiors, $D_1 = \cup_{i=1}^m N_i,\; N_i^{\circ} \cap N_j^\circ = \emptyset$ for $i \ne j$, and consider a mapping $f : D_1 \to \RRR^n$ such that the restriction $f\rfloor_{N_i}$ for each $i$ defines a diffeomorphism between $N_i$ and $f(N_i)$, and the interiors of the images $f(N_i),\; i = 1,\ldots,m$ are mutually disjoint and do not intersect $D_1$. We take $D_2 = \cup_{i=1}^m f(N_i)$. As above, define the mapping $g$ by $g(x) = f(x) - x$.

Let us additionally assume that the restriction of $g$ on $N_i$ is the gradient of a smooth function $G_i : N_i \to \RRR$,\, $g_i := g\rfloor_{N_i} = \nabla G_i$. Further, assume that each function $G_i$ can be extended to a smooth function $\tilde G_i$ defined on a larger domain $\tilde N_i \supset N_i$ and such that the mapping $\tilde f_i$ defined by $\tilde f_i(x) = x + \nabla \tilde G_i(x)$ is a diffeomorphism from $\tilde N_i$ to $\tilde f_i(\tilde N_i)$ and, moreover, $\tilde N_i$ and $\tilde f_i(\tilde N_i)$ are convex and their interiors are disjoint, $\tilde N_i^{\circ} \cap \tilde f_i(\tilde N_i)^{\circ} = \emptyset$.

Note that the interiors of $\tilde N_i,\; i = 1,\ldots,m$ do not need to be disjoint.

\begin{predl}\label{cor}
Let the above assumptions be satisfied. Then $f$ can be realized by reflections from a finite collection of smooth mirrors, where each light ray makes exactly 2 reflections.
\end{predl}

\begin{proof}
The collection of mirrors we are looking for is a finite collection of 2-mirror systems. The first mirror of  $i$th system is the graph of a function $\Phi_1^i = \frac{1}{c_i}\, G_i + h_i$ defined on $N_i$, while the second one is the graph of the function
\beq\label{psi2}
\Phi_2^i(x + g_i(x)) = \Phi_1^i(x) +\frac{|g_i(x)|^2 - c_i^2}{2c_i}
\eeq
defined on $f(N_i)$ (compare with formula (\ref{Psi})).

These mirrors define a family of 3-segment polygonal lines having billiard reflections from the mirrors at the respective points. It remains to properly choose the parameters $c_i$ and $h_i$ to guarantee  that these polygonal lines do not have superfluous intersections with the mirrors, and therefore are {\it true} billiard trajectories.

Since the interiors of the sets $N_i$ and $f(N_i),\; i = 1,\ldots,m$ are mutually disjoint, one concludes that the first and the last segments do not have superfluous (others than at their endpoints) intersections. It remains to check the intermediate segments. Below we prove that the intermediate segments of  $i$th system do not have irrelevant intersections (a) with the mirrors of the same system and (b) with the mirrors of the other systems.

(a) By the above assumption we can extend the two mirrors in  $i$th system. The extended first mirror is the graph of the function $\tilde \Phi_1^i(x) = \frac{1}{c_i} \tilde G_i(x) + h_i,\; x \in \tilde N_i$. The extended second mirror is the graph of the function $\tilde \Phi_2^i(x)$ defined on $\tilde f_i(\tilde N_i)$ by the formula analogous to (\ref{psi2}), with $\Phi_1^i$ and $g_i$ replaced by $\tilde\Phi_1^i$ and $\tilde g_i$ (with $\tilde g_i(x) = \tilde f_i(x) - x$). This extended 2-mirror system generates a bundle of 3-segment polygonal lines having billiard reflections off the two mirrors. As follows from statement (b) of Theorem \ref{t1}, for $c_i$ sufficiently large, the intermediate segments of the polygonal lines do not have interior points of intersection with the mirrors of the extended $i$th system. The same is obviously true for the original systems formed by the graphs of $\Phi_1^i$ and $\Phi_2^i$.

(b) Choose $h_i$ inductively, so that  $i$th pair of mirrors lies below the previous ($1,\ldots,i-1$) pairs. This choice guarantees that intermediate segments in  $i$th system do not intersect the mirrors of $j$th systems with $j \ne i$.
\end{proof}

\section{Mirror transformations in $\RRR^3$}
\label{sec:mirr_trans}

Let $f : D_1 \to D_2$ be a diffeomorphism of compact domains in $\RRR^2$. Consider the question: is it possible to realize $f$ by mirror reflections, and how many mirrors and/or mirror reflections are needed? Theorem \ref{t1} states that if $f$ is a gradient mapping, then (under some additional assumptions) a realization with 2 reflections is possible, and if $f$ is not a gradient mapping, such a realization is impossible.

Taking into account Proposition \ref{cor}, one concludes, in a similar way, that if $f$ is a {\it piecewise gradient} mapping, then (again under some additional assumptions) a realization with 2 reflections from {\it finitely many} mirrors is possible, and if $f$ is not piecewise-gradient, such a realization with 2 reflections is impossible.

It is natural to try to represent $f$ as the composition of two gradient diffeomorphisms. In this case the following lemma holds.

\begin{lemma}\label{l1}
Let $f = \nabla\psi \circ \nabla\varphi$, where $\varphi : D_1 \to \RRR$ and $\psi : D' \to \RRR$ are smooth functions, $\nabla\varphi : D_1 \to D'$ and $\nabla\psi : D' \to D_2$ are diffeomorphisms, and $D_1,\, D',\, D_2$ are convex compact domains. Then $f$ can be realized by 4 reflections from 4 smooth mirrors.
\end{lemma}

\begin{proof}
Without loss of generality one can assume that $D'$ does not intersect $D_1 \cup D_2$. Indeed, otherwise one can replace $\varphi$ by $\varphi + bx$,\, $\psi$ by $\psi - bx$, and $D'$ by $D' + b$, where $b \in \RRR^2$ is chosen so that $D' + b$ does not intersect $D_1 \cup D_2$.

By claim (b) of Theorem \ref{t1}, there are two smooth mirrors transforming the parallel bundle of rays with the orthogonal cross section $D_1$ into the parallel bundle with the orthogonal cross section $D'$, and two smooth mirrors transforming the parallel bundle with the cross section $D'$ into the parallel bundle with the cross section $D_2$. Further, by parallel shifting in the vertical direction $(0,0,1)$ one can ensure that the former pair of mirrors lies in the half-space $z<0$ and the latter pair of mirrors lies in the half-space $z>0$, and therefore they do not intersect (see Fig.~\ref{fig_4refl}).
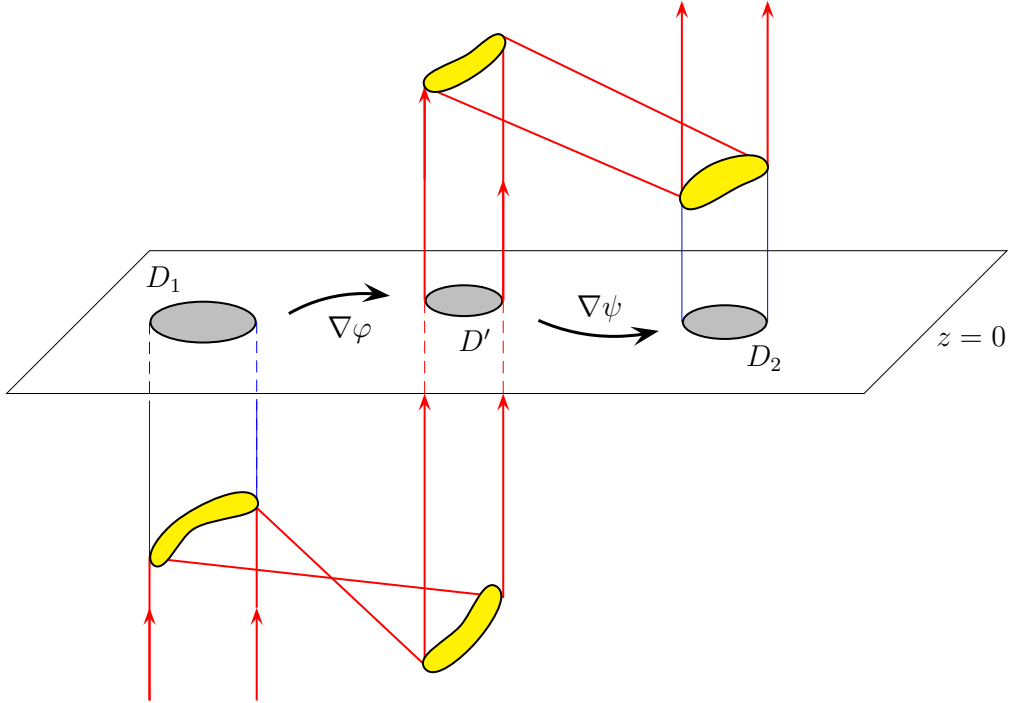
\begin{figure}[h]
\begin{picture}(0,270)
\rput(6,0.4){
\scalebox{0.95}{
   \pspolygon[linewidth=0.4pt](-6,4)(-4,6)(8,6)(6,4)
   \psellipse[fillstyle=solid,fillcolor=lightgray](-3.25,5)(0.75,0.3)
   \psellipse[fillstyle=solid,fillcolor=lightgray](0.4,5.3)(0.55,0.23)
   \psline[linewidth=0.8pt,linecolor=red,arrows=->,arrowscale=1.6](-4,-0.3)(-4,1.0)
  \psline[linewidth=0.8pt,linecolor=red,arrows=->,arrowscale=1.6](-4,-0.3)(-4,1.0)(-4,1.7)(-3.75,1.67)
  \psline[linewidth=0.8pt,linecolor=red,arrows=->,arrowscale=1.6](-3.75,1.67)(0.95,1.15)(0.95,4)
  \psline[linewidth=0.3pt,linecolor=red,linestyle=dashed](0.95,4)(0.95,5.3)
  \psline[linewidth=0.8pt,linecolor=red,arrows=->,arrowscale=1.6](0.95,5.3)(0.95,7)
            \psline[linewidth=0.8pt,linecolor=red,arrows=->,arrowscale=1.6](0.95,5.3)(0.95,7)(0.95,9)(4.65,7.2)(4.65,9.5)
                \psline[linewidth=0.4pt,linecolor=blue](4.65,7.2)(4.65,5)
                \psline[linewidth=0.4pt,linecolor=blue](3.45,6.75)(3.45,5)
                \psellipse[fillstyle=solid,fillcolor=lightgray](4.05,5)(0.6,0.25)
       \psline[linewidth=0.4pt,linecolor=blue](-4,1.7)(-4,3.9)
       \psline[linewidth=0.4pt,linecolor=blue,linestyle=dashed](-4,4)(-4,5)
   \psline[linewidth=0.8pt,linecolor=red,arrows=->,arrowscale=1.6](-2.5,-0.3)(-2.5,1.0)
   \psline[linewidth=0.8pt,linecolor=red,arrows=->,arrowscale=1.6](-2.5,1.0)(-2.5,2.4)(-0.15,0.2)(-0.15,4)
   \psline[linewidth=0.3pt,linecolor=red,linestyle=dashed](-0.15,4)(-0.15,5.3)
   \psline[linewidth=0.8pt,linecolor=red,arrows=->,arrowscale=1.6](-0.15,5.3)(-0.15,8.3)
            \psline[linewidth=0.8pt,linecolor=red,arrows=->,arrowscale=1.6](-0.15,7)(-0.15,8.3)(3.45,6.75)(3.45,9.5)
            \psecurve[fillstyle=solid,fillcolor=yellow](0.5,8.5)(-0.15,8.3)(0.45,8.75)(0.95,9)(0.55,8.45)(-0.15,8.3)(0.4,8.8)
                 \psecurve[fillstyle=solid,fillcolor=yellow](4.25,6.9)(3.45,6.65)(3.85,7.2)(4.65,7.2)(4.25,6.9)(3.45,6.65)(4.15,7.1)
  \psline[linewidth=0.4pt,linecolor=blue](-2.5,2.5)(-2.5,3.9)
  \psline[linewidth=0.4pt,linecolor=blue,linestyle=dashed](-2.5,2.4)(-2.5,5)
   \rput(-3.55,-0.3){\psecurve[fillstyle=solid,fillcolor=yellow](4.1,0.8)(3.4,0.45)(3.9,1.05)(4.4,1.6)(4.1,0.8)(3.4,0.45)(3.9,1.05)}
   \rput(-3.55,-0.8){\psecurve[fillstyle=solid,fillcolor=yellow](0.35,2.8)(-0.4,2.4)(-0.1,3.0)(1.06,3.3)(0.16,2.9)(-0.4,2.4)(-0.1,3.0)}
   \rput(-3.8,5.6){$D_1$}
   \rput(0.55,4.75){$D'$}
   \rput(4.6,4.5){$D_2$}

    \rput(7.5,4.8){$z = 0$}
   \rput(-1.2,4.9){$\nabla\varphi$}
   \psarc[linewidth=1.2pt,arrows=<-,arrowscale=2](-1,3.3){2.1}{80}{120}
   \rput(2.3,5.2){$\nabla\psi$}
   \psarc[linewidth=1.2pt,arrows=->,arrowscale=2](2.5,7.4){2.6}{246}{284}
}
}
\end{picture}
\caption{A 4-mirror transformation of light rays.}
\label{fig_4refl}
\end{figure}

Note that in this construction the intermediate (after two reflections) bundle of rays is also parallel and has vertical direction $(0,0,1)$.
\end{proof}

However, in general, one cannot represent a local plane diffeomorphism as a composition of two gradient diffeomorphisms. In the next subsection we prove the following more restricted result: any {\it orientation reversing} diffeomorphism can be locally (in a neighborhood of any point) represented as a composition of two gradient diffeomorphisms. For {\it orientation preserving} diffeomorphisms, however, this is not always true, as an example will show.

\subsection{Orientation reversing transformations}

\begin{predl}\label{l2}
Any orientation reversing diffeomorphism $f : D_1 \to D_2$ is locally (in a neighborhood of any point $x \in D_1$) a composition of two gradient diffeomorphisms.
\end{predl}

\begin{proof}
Since the diffeomorphism $f : (x_1,x_2) \mapsto (f_1(x_1,x_2), f_2(x_1,x_2))$ is orientation reversing, one has
$$
\left| \begin{array}{cc}
\frac{\pl f_1}{\pl x_1} & \frac{\pl f_1}{\pl x_2}\\ \frac{\pl f_2}{\pl x_1} & \frac{\pl f_2}{\pl x_2}
\end{array}\right| < 0.
$$

Fix $x \in D_1$. We are looking for a small open disc $B(x)$ centered at $x$ and for two functions $\varphi$ and $\psi$ such that $\nabla\varphi$ is a diffeomorphism from $B(x)$ onto an open set $B'$,\, $\nabla\psi$ is a diffeomorphism from $B'$ onto $f(B(x))$, and
$$
f\rfloor_{B(x)} = \nabla\psi \circ \nabla\varphi.
$$
Using the fact that the inverse of a gradient diffeomorphism is again a gradient diffeomorphism and setting $(\nabla\psi)^{-1}=\nabla u$, one can reformulate the task as follows: find a disc $B(x)$ and two functions $\vphi : B(x) \to \RRR$ and $u : f(B(x)) \to \RRR$ such that $\nabla\varphi : B(x) \to B' := \nabla\varphi(B(x))$ and $\nabla u : f(B(x)) \to B'$ are diffeomorphisms and
\beq\label{eq1}
\nabla u \circ f\rfloor_{B(x)} = \nabla\varphi.
\eeq
(see Fig.~\ref{fig_comp}).

\begin{figure}[h]
\begin{picture}(0,120)
\rput(6,0.5){
\scalebox{1}{
\psellipse(-3.7,2)(2,1.5)
\psellipse(4.2,2)(2,1.5)
\pscircle[fillstyle=solid,fillcolor=lightgray](-4,2){0.4}
\pscircle[fillstyle=solid,fillcolor=lightgray](4,2){0.4}
\pscircle[fillstyle=solid,fillcolor=lightgray](0,0){0.4}
\psdots(-4,2)(4,2)(0,0)
\psarc[linewidth=1pt,arrows=<-,arrowscale=2](0,-6){9}{70}{110}
\psarc[linewidth=1pt,arrows=<-,arrowscale=2](-5,-5){7}{50}{75}
\psarc[linewidth=1pt,arrows=<-,arrowscale=2](-1,7){7}{283}{308}
\rput(-4.3,2.7){$B(x)$}
\rput(4.6,2.7){$f(B(x))$}
\rput(-5.1,0.5){$D_1$}
\rput(6.3,1.1){$D_2$}
\rput(0,3.3){$f$}
\rput(-1.5,0.7){$\nabla\varphi$}
\rput(1.8,0.25){$\nabla u$}
}
}
\end{picture}
\caption{Local representation of $f$ as a composition of two gradient diffeomorphisms.}
\label{fig_comp}
\end{figure}
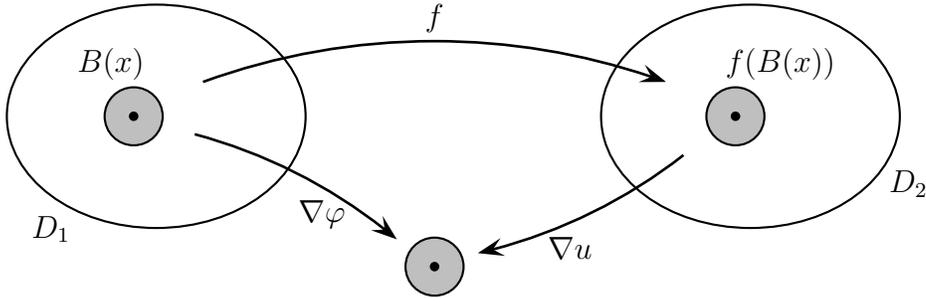

Note that the mapping on the left hand side of (\ref{eq1}) takes $x = (x_1,x_2)$ to $(\pl_1 u(f_1(x),f_2(x)),\, \pl_2 u(f_1(x),f_2(x)))$. This mapping is the gradient of a function (and therefore (\ref{eq1}) is solvable), if and only if
$$
\frac{\pl}{\pl x_2} \big[\pl_1 u(f_1(x),f_2(x)) \big] = \frac{\pl}{\pl x_1} \big[\pl_2 u(f_1(x),f_2(x)) \big],
$$
or, in a more detailed form,
$$
\pl^2_{11} u(f_1(x),f_2(x))\, \frac{\pl f_1}{\pl x_2}(x) + \pl^2_{12} u(f_1(x),f_2(x))\, \frac{\pl f_2}{\pl x_2}(x) \hspace*{40mm}
$$
\beq\label{edp0}
\hspace*{15mm} = \pl^2_{21} u(f_1(x),f_2(x))\, \frac{\pl f_1}{\pl x_1}(x) + \pl^2_{22} u(f_1(x),f_2(x))\, \frac{\pl f_2}{\pl x_1}(x).
\eeq
Denote by $\xi \mapsto g(\xi)$ the mapping inverse to $x \mapsto f(x)$; then we arrive at the equation for the unknown function $u$
\beq\label{pde}
\frac{\pl f_1}{\pl x_2}(g(\xi))\, \frac{\pl^2 u}{\pl\xi_1^2}(\xi) + \Big( \frac{\pl f_2}{\pl x_2}(g(\xi)) - \frac{\pl f_1}{\pl x_1}(g(\xi)) \Big)\, \frac{\pl^2 u}{\pl\xi_1 \pl\xi_2}(\xi) - \frac{\pl f_2}{\pl x_1}(g(\xi))\, \frac{\pl^2 u}{\pl\xi_2^2}(\xi) = 0.
\eeq
The discriminant of this equation is
$$
-4 \frac{\pl f_1}{\pl x_2} \frac{\pl f_2}{\pl x_1} - \Big( \frac{\pl f_2}{\pl x_2} - \frac{\pl f_1}{\pl x_1} \Big)^2
= 4\left| \begin{array}{cc}
\frac{\pl f_1}{\pl x_1} & \frac{\pl f_1}{\pl x_2}\\ \frac{\pl f_2}{\pl x_1} & \frac{\pl f_2}{\pl x_2}
\end{array}\right|
- \Big( \frac{\pl f_2}{\pl x_2} + \frac{\pl f_1}{\pl x_1} \Big)^2 < 0;
$$
that is, the  2nd order PDE  (\ref{pde}) is hyperbolic and hence has infinitely many solutions (note that the boundary conditions are not yet specified).

Consider two characteristics of this equation through $x$; a solution $u$ in a neighborhood of $x$ is uniquely defined by its values on these characteristics, again in a neighborhood of $x$. Let us show that the first and the second derivatives at $x$ along the characteristics can be chosen in such a way that the gradient $\nabla u$ of the resulting solution is a diffeomorphism in a sufficiently small disc  $B(x)$.

Indeed, let the characteristics be given by $\al(\xi_1,\xi_2) =$ const and $\bt(\xi_1,\xi_2) =$ const. Let the first derivatives at $x$  be zero, $u_\al = 0$ and $u_\bt = 0$; then the second derivatives of $u$ in the coordinate systems $\xi_1,\, \xi_2$ and $\al,\, \bt$ are related as follows:
\begin{equation}\label{matrix}
Hu_{\xi_1,\xi_2} = J^T Hu_{\al,\bt}\, J,
\end{equation}
where
$$
Hu_{\xi_1,\xi_2} =
\left(\!\!
\begin{array}{cc}
\partial_{11}u & \partial_{12}u\\
\partial_{12}u & \partial_{22}u
\end{array}
\!\! \right), \,
Hu_{\al,\bt} =
\left(\!\!
\begin{array}{cc}
u_{\al\al} & u_{\al\bt}\\
u_{\al\bt}  & u_{\bt\bt}
\end{array}
\!\! \right),
\ \text{and}\ \,
 J = \left(\!\!
\begin{array}{cc}
\partial_{1}\al & \partial_{2}\al\\
\partial_{1}\bt & \partial_{2}\bt
\end{array}
\!\! \right).
$$
Taking into account that the Jacobi matrix $J$ is non-degenerate, it suffices to define the second derivatives along the characteristics $u_{\al\al},\, u_{\al\bt}$, and $u_{\bt\bt}$ in such a way that the corresponding matrix $Hu_{\al,\bt}$ is non-degenerate. Then the Hessian matrix $Hu_{\xi_1,\xi_2}$ in the original coordinate system is also non-degenerate, and therefore $\nabla u$ is a diffeomorphism in a neighborhood of the point.

We are now guaranteed that the composition of mappings on the left hand side of (\ref{eq1}) is a diffeomorphism and the gradient of a certain function.
\end{proof}

The following statement is a corollary of Proposition \ref{l2}.

\begin{corollary}\label{cor1}
An orientation reversing diffeomorphism can be locally realized by 4 reflections from a 4-mirror system.
\end{corollary}

\begin{proof}
Indeed, by Proposition \ref{l2}, for any point $x \in D_1$ we have a representation
$$
f\rfloor_{B(x)} = \nabla\psi \circ \nabla\varphi,
$$
where $B(x)$ is a disc centered at $x$. The sets $f(B(x))$ and $\nabla\varphi(B(x))$ are deformed ellipses. They become convex when the radius of $B(x)$ is sufficiently small. Further, by replacing, if necessary, $\varphi$ by $\tilde\varphi(x) = \varphi(x) + bx$ and $\psi$ by $\tilde\psi(x) = \psi(x) - bx$, where $b \in \RRR^2$ is an suitable vector, one can assure that $B' := \nabla\varphi(B(x)) + b$ does not intersect $B(x) \cup f(B(x))$.

Now we have two gradient diffeomorphisms $\nabla\tilde\varphi : B(x) \to B'$, where $B(x)$ and $B'$ are convex and disjoint, and $\nabla\tilde\psi : B' \to f(B(x))$, where $B'$ and $f(B(x))$ are convex and disjoint, and $f\rfloor_{B(x)}$ is their composition. By statement (b) of Theorem \ref{t1}, both $\nabla\tilde\varphi$ and $\nabla\tilde\psi$ can be realized by 2 reflections in two 2-mirror systems. By a vertical (along the $z$-axis) shift one can render  these systems  disjoint. As a result, the union of these mirrors locally realizes $f$ via 4 reflections.
\end{proof}

Without the assumption of orientation reversal the statement of Proposition \ref{l2} is not true: there exist a diffeomorphism $f : D_1 \to D_2$ (which preserves the orientation) and a point $x \in D_1^\circ$ such that $f$ is not a composition of two gradient diffeomorphisms in a neighborhood of $x$. Below we reproduce (in a slightly modified form) an example proposed by A. Glutsyuk.

\begin{primer}\label{primer}\rm
Consider the diffeomorphism
$$
f(x_1,x_2) = e^{x_2}(x_1,x_2)
$$
defined on a ball $\bar B_r(0,0)$ of radius $r < 1$, and take $x = (0,0)$. Suppose that $f$ is a composition of two gradient diffeomorphisms in a neighborhood of $(0,0)$; then, repeating the argument in the proof of Proposition \ref{l2}, one arrives at formula (\ref{edp0}) which, in this particular case, takes the form
$$
x_1 e^{x_2}\, \partial^2_{11} u(x_1 e^{x_2},\, x_2 e^{x_2}) + x_2 e^{x_2}\, \partial^2_{12} u(x_1 e^{x_2},\, x_2 e^{x_2}) = 0,
$$
again in a neighborhood of $(0,0)$. Denoting $\xi = x_1 e^{x_2}$,\, $\eta = x_2 e^{x_2}$ and  $f(\xi, \eta) = u'_\xi(\xi, \eta)$, one obtains the equation
$$
\xi f'_\xi + \eta f'_\eta = 0
$$
in a neighborhood of $(0,0)$. By Euler's formula, it follows that $f(\xi, \eta)$ is homogeneous of degree zero, and hence  $f =$ const. This implies that $\nabla u(x_1,x_2)$ has a constant first component, and therefore $\nabla u$ is not a diffeomorphism.
\end{primer}

Now, using Propositions \ref{cor} and \ref{l2}, we shall prove the following theorem. The main concern here, like in Proposition \ref{cor}, is to avoid irrelevant intersections of polygonal lines (presumably billiard trajectories) with mirrors.

\begin{theorem}\label{t2}
Any orientation reversing diffeomorphism $f : D_1 \to D_2$ of compact domains in $\RRR^2$ can be realized by 4 reflections from a finite collection of mirrors.
\end{theorem}

\begin{proof}
By Proposition \ref{l2}, $f$ can be locally represented as a composition of two gradient diffeomorphisms. Equivalently, for any $x \in D_1$, there exists an open disc $B(x)$ and two functions $\varphi_{(x)}$ and $u_{(x)}$ defined on $B(x)$ and $f(B(x))$, respectively, such that
\beq\label{eq2}
\nabla u_{(x)} \circ f\rfloor_{B(x)} = \nabla\varphi_{(x)},
\eeq
and $\nabla\varphi_{(x)} : B(x) \to B'_{(x)} := \nabla\varphi_{(x)}(B(x))$ and $\nabla u_{(x)} : f(B(x)) \to B'_{(x)}$ are diffeomorphisms.

Note that $f(B(x))$ and $B'_{(x)}$ are deformed ellipses that become convex when the radius of $B(x)$ is sufficiently small. Thus, without loss of generality, one can assume that both these domains are convex for all $x$.

Take a finite subcover $\{ B(x_k) \}$ of $D_1$ and choose vectors $b_k \in \RRR^2$ such that $B'_{(x_k)} + b_k =: B'_k$ are mutually disjoint and do not intersect the sets $\cup_k B(x_k)$ and $\cup_k f(B(x_k))$.

Choose closed domains $N_k \subset B(x_k)$ so that their interiors $N_k^\circ$ are mutually disjoint and $\cup_k N_k = D_1$. Let $\varphi_k$ be the restriction of the mapping $x \mapsto \varphi_{(x_k)}(x) + b_k x$ to $N_k$, and $u_k$ be the restriction of the mapping $x \mapsto u_{(x_k)}(x) + b_k x$ to $f(N_k)$.

Define the piecewise smooth mappings $\varphi$ on $D_1$ and $u$ on $D_2$ by $\varphi\rfloor_{N_k^\circ} = \varphi_k$ and $u\rfloor_{f(N_k^\circ)} =u_k$. Then we have
$$
\nabla u \circ f = \nabla\varphi.
$$
Moreover, $\nabla\varphi_k$ can be extended to a diffeomorphism between disjoint convex sets $B(x_k)$ and $B'_k$, and $\nabla u_k$ can be extended to a diffeomorphism between disjoint convex sets $f(B(x_k))$ and $B'_k$. Thus, both  mappings $\nabla\varphi$ and $\nabla u$ satisfy the assumptions of Corollary \ref{cor} and therefore can be realized by finite collections of mirrors, with each realization involving two reflections. By simultaneous vertical shifting of the mirrors one ensures that both these collections (denoted by $\mathcal{C}_\vphi$ and $\mathcal{C}_u$) lie in the lower half-space $z < 0$.

The inverse of the gradient diffeomorphism $\nabla u$ is again a gradient diffeomorphism,
$$
(\nabla u)^{-1} = \nabla\psi,
$$
and it can be realized by the collection of mirrors $\mathcal{C}_\psi$ symmetric to $\mathcal{C}_u$ with respect to the plane $z = 0$ (and therefore $\mathcal{C}_\psi$ lies in the upper half-space $z > 0$).

We have
$$
f = \nabla\psi \circ \nabla\varphi;
$$
therefore $f$ can be realized by reflections from the mirrors in $\mathcal{C}_\vphi$ and $\mathcal{C}_\psi$ via 4 reflections (see Fig.~\ref{fig_4tr}).

\begin{figure}[h]
\begin{picture}(0,310)
\rput(6,0.5){
\scalebox{0.95}{
   \pspolygon[linewidth=0.4pt](-6,4)(-4,6)(8,6)(6,4)
   \rput(-0.6,6.4){\scalebox{1.2}{$f$}} 
   \psarc[linewidth=1pt,arrows=<-,arrowscale=2](0.6,-2.6){8.8}{64}{114} 
      \psellipse[fillstyle=solid,fillcolor=lightgray](-3.5,5)(1,0.3) 
       \psellipse[fillstyle=solid,fillcolor=lightgray](5.2,5)(1,0.3) 
        \psellipse[fillstyle=solid,fillcolor=lightgray](0.15,5.2)(0.35,0.15) 
         \psellipse[fillstyle=solid,fillcolor=lightgray](2.05,5.3)(0.25,0.15)
              \psellipse[fillstyle=solid,fillcolor=lightgray](1.3,4.7)(0.3,0.125) 
          \psline[linewidth=0.4pt,linecolor=blue,linestyle=dashed](-3.8,4)(-3.8,4.7) 
          \psline[linewidth=0.4pt,linecolor=blue,linestyle=dashed](-4.5,4)(-4.5,5) 
          \psline[linewidth=0.4pt,linecolor=blue,linestyle=dashed](-3,4)(-3,4.7)
      \psline[linewidth=0.4pt,linecolor=blue](-4.5,2.7)(-4.5,3.9) 
      \psline[linewidth=0.4pt,linecolor=blue](-3.8,3.3)(-3.8,3.9) 
            \psline[linewidth=0.4pt,linecolor=blue](-3,1.8)(-3,3.9)
              \psline[linewidth=0.4pt,linecolor=blue](-2.5,0.6)(-2.5,3.9)
  \psline[linewidth=0.4pt,linecolor=blue,linestyle=dashed](-2.5,4)(-2.5,5)
 \psecurve[fillstyle=solid,fillcolor=yellow,linewidth=0.4pt](-4.2,2.85)(-4.5,2.7)(-4.1,3.15)(-3.8,3.3)(-4.2,2.85)(-4.5,2.7)(-4.1,3.15)  
\rput(4.3,-0.7) {\psecurve[fillstyle=solid,fillcolor=yellow,linewidth=0.4pt](-4.2,2.85)(-4.5,2.7)(-4.1,3.15)(-3.8,3.3)(-4.2,2.85)(-4.5,2.7)(-4.1,3.15)}
   \psline[linewidth=0.8pt,linecolor=red,arrows=->,arrowscale=1.6](-4.5,-0.7)(-4.5,0) 
   \psline[linewidth=0.8pt,linecolor=red,arrows=->,arrowscale=1.6](-4.5,0)(-4.5,2.7)(-0.2,2) 
    \psline[linewidth=0.8pt,linecolor=red,arrows=->,arrowscale=1.6](-3.8,-0.7)(-3.8,0) 
     \psline[linewidth=0.8pt,linecolor=red,arrows=->,arrowscale=1.6](-3.8,0)(-3.8,3.3)(0.5,2.6) 
    \psline[linewidth=0.8pt,linecolor=red,arrows=->,arrowscale=1.6](-0.2,2)(-0.2,4) 
   \psline[linewidth=0.4pt,linecolor=red,linestyle=dashed](-0.2,4)(-0.2,5.2)    
  \psline[linewidth=0.8pt,linecolor=red,arrows=->,arrowscale=1.6](0.5,2.6)(0.5,4) 
     \rput(0,0){\psecurve[fillstyle=solid,fillcolor=yellow,linewidth=0.4pt](1.2,1)(1,0.8)(1.3,1.2)(1.6,1.3)(1.2,1)(1,0.8)(1.3,1.2)} 
     \rput(0.8,-0.85){\psecurve[fillstyle=solid,fillcolor=yellow,linewidth=0.4pt](1.2,0.8)(1,0.8)(1.3,1.1)(1.55,1.15)(1.2,0.8)(1,0.8)(1.3,1.1)} 
      \psline[linewidth=0.8pt,linecolor=red,arrows=->,arrowscale=1.6](-3.8,1.3)(1,0.8)(1,4)
      \psline[linewidth=0.8pt,linecolor=red,arrows=->,arrowscale=1.6](-3,-0.7)(-3,0) 
      \psline[linewidth=0.8pt,linecolor=red,arrows=->,arrowscale=1.6](-3,0)(-3,1.8)(1.6,1.3)(1.6,4) 
  \psline[linewidth=0.4pt,linecolor=red,linestyle=dashed](0.5,4)(0.5,5.2)   
  \psline[linewidth=0.4pt,linecolor=red,linestyle=dashed](1,4)(1,4.7) 
    \psline[linewidth=0.4pt,linecolor=red,linestyle=dashed](1.6,4)(1.6,4.7)
        \psline[linewidth=0.8pt,linecolor=red,arrows=->,arrowscale=1.6](-3,0.3)(1.8,0)(1.8,4)
        \psline[linewidth=0.8pt,linecolor=red,arrows=->,arrowscale=1.6](-2.5,-0.7)(-2.5,0) 
        \psline[linewidth=0.8pt,linecolor=red,arrows=->,arrowscale=1.6](-2.5,0)(-2.5,0.6)(2.3,0.3)(2.3,4) 
          \psline[linewidth=0.4pt,linecolor=red,linestyle=dashed](1.8,4)(1.8,5.3)  
  \psline[linewidth=0.4pt,linecolor=red,linestyle=dashed](2.3,4)(2.3,5.3) 
  \rput(-4.8,0.6){\psecurve[fillstyle=solid,fillcolor=yellow,linewidth=0.4pt](1.2,1)(1,0.7)(1.3,1.2)(1.8,1.3)(1.2,0.8)(1,0.7)(1.3,1.2)} 
\rput(-3.8,-0.5){\psecurve[fillstyle=solid,fillcolor=yellow,linewidth=0.4pt](1.0,0.85)(0.8,0.85)(1.0,1.15)(1.33,1.23)(1.0,0.85)(0.8,0.85)(1.0,1.15)} 
    \rput(0,-0.4){\psecurve[fillstyle=solid,fillcolor=yellow,linewidth=0.4pt](0.25,8.5)(-0.18,8.4)(0.2,8.6)(0.475,8.8)(0.25,8.5)(-0.18,8.4)(0.2,8.6)}
       \psline[linewidth=0.8pt,linecolor=red,arrows=->,arrowscale=1.6](-0.2,5.2)(-0.2,7)
       \psline[linewidth=0.8pt,linecolor=red,arrows=->,arrowscale=1.6](-0.2,6.6)(-0.2,8)(4.2,6.75)(4.2,11)
  \psline[linewidth=0.8pt,linecolor=red,arrows=->,arrowscale=1.6](0.5,5.2)(0.5,7)
  \psline[linewidth=0.8pt,linecolor=red,arrows=->,arrowscale=1.6](0.5,7)(0.5,8.3)(5,7.2)(5,11)
                 \psline[linewidth=0.4pt,linecolor=blue](5,7.2)(5,5.3) 
                \psline[linewidth=0.4pt,linecolor=blue](4.2,6.75)(4.2,5)      
       \psline[linewidth=0.4pt,linecolor=blue](5.6,9.2)(5.6,5.25) 
              \psline[linewidth=0.4pt,linecolor=blue](6.2,9.5)(6.2,5) 
                   \psline(5,4.7)(4.6,5.2)
              \psline(5.7,4.75)(5.2,5.3)
  \rput(2.0,2.55){\scalebox{0.64}       {\psecurve[fillstyle=solid,fillcolor=yellow](4.25,6.9)(3.45,6.65)(3.85,7.2)(4.65,7.2)(4.25,6.9)(3.45,6.65)(4.15,7.1)}} 
   \psline[linewidth=0.8pt,linecolor=red,arrows=->,arrowscale=1.6](1,4.7)(1,8.5)(5,8.5)(5,11) 
  \psline[linewidth=0.8pt,linecolor=red,arrows=->,arrowscale=1.6](1.6,4.7)(1.6,9)(5.6,9)(5.6,11) 
    \psecurve[fillstyle=solid,fillcolor=yellow,linewidth=0.4pt](5.3,8.65)(5,8.5)(5.3,8.85)(5.6,9)(5.3,8.65)(5,8.5)(5.3,8.85) 
     \rput(-4.0,0)
     {\psecurve[fillstyle=solid,fillcolor=yellow,linewidth=0.4pt](5.3,8.65)(5,8.5)(5.3,8.85)(5.6,9)(5.3,8.65)(5,8.5)(5.3,8.85)} 
   \psline[linewidth=0.8pt,linecolor=red,arrows=->,arrowscale=1.6](1.8,5.3)(1.8,9.5)(5.6,9.2)(5.6,11)
  \psline[linewidth=0.8pt,linecolor=red,arrows=->,arrowscale=1.6](2.3,5.3)(2.3,10)(6.2,9.7)(6.2,11)
  \psecurve[fillstyle=solid,fillcolor=yellow,linewidth=0.4pt](6,9.3)(5.6,9.2)(5.85,9.56)(6.25,9.67)(6,9.3)(5.6,9.2)(5.85,9.56)
     \pscurve[fillstyle=solid,fillcolor=yellow,linewidth=0.4pt](2.05,9.7)(1.8,9.5)(2.0,9.85)(2.3,10)(2.05,9.7)(1.8,9.5)(2.0,9.85)
       \rput(-3.8,5.7){$D_1$}
\psline(-3.8,4.7)(-4.0,5.26)
\psline(-3,4.73)(-3.2,5.28)
   \rput(-4.25,4.5){$N_1$}
      \rput(-3.35,4.4){$N_2$}
   \rput(-2.72,4.35){$N_3$}
   \rput(4.9,4.3){$D_2$}
    \rput(7.5,4.8){$z = 0$}
   \rput(-1.2,4.9){$\nabla\varphi$}
   \psarc[linewidth=1pt,arrows=<-,arrowscale=2](-0.9,3.3){2.1}{81}{120}
   \rput(3.3,5.2){$\nabla u$}
   \rput(-5.1,1.3){\scalebox{1.3}{$\mathcal{C}_\varphi$}}
   \rput(0.4,9.4){\scalebox{1.3}{$\mathcal{C}_\psi$}}
   \psarc[linewidth=1pt,arrows=<-,arrowscale=2](3.5,7.4){2.6}{250}{284}
}}
\end{picture}
\caption{A representation of a diffeomorphism $f$ by mirror reflections.}
\label{fig_4tr}
\end{figure}
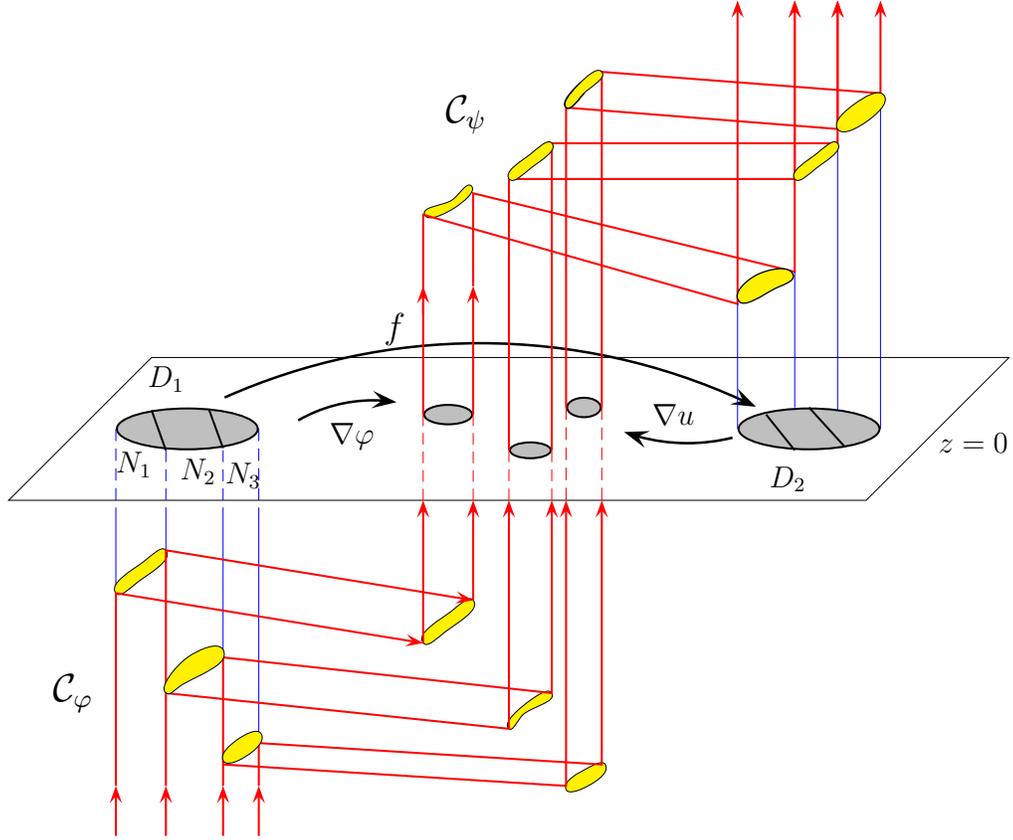

\end{proof}

\subsection{Further results}

We derive several corollaries from Theorem \ref{t2}.

\begin{corollary}\label{cor2}
Any orientation preserving diffeomorphism $f : D_1 \to D_2$ of compact domains in $\RRR^2$ can be realized by 6 reflections from a finite collection of mirrors.
\end{corollary}

\begin{proof}
Without loss of generality, assume that $D_1$ lies in the half-plane $x_1 < 0$. The mapping $f$ is the composition  $f = f^* \circ \sigma$, the second mapping defined by
$$
\sigma(x_1,x_2) = (-x_1,x_2),
$$
and the first one by
$$
f^*(x_1,x_2) = f(-x_1,x_2).
$$
The mapping $\sigma$ is realized by 2 reflections from two pieces of parabolic cylinders, $z = (c^2 - x_1^2)/(2c)$ and $z = (x_1^2 - c^2)/(2c)$, with arbitrary $c > 0$, the former piece lying in the half-plane $x_1 < 0$ and the latter one in the half-plane $x_1 > 0$. The mapping $f^*$ is orientation reversing and, by Theorem \ref{t2}, it  can be realized by 4 reflections from finitely many mirrors. One easily excludes the possibility of superfluous intersections. As a result, one obtains a realization of $f$ with 6 reflections.
\end{proof}

Consider two generic germs of normal families of oriented lines in $\RRR^3$, and let $\Sigma_1$ and $\Sigma_2$ be normal surfaces (wave fronts) of these families, topologically, discs.
Consider a diffeomorphism $f$ from the first normal family to the second one; $f$ can be thought of as a diffeomorphism $\Sigma_1 \to \Sigma_2$.

We wish to realize $f$ as a composition of mirror reflections.

\begin{corollary}\label{cor3}
A diffeomorphism $f$ can be realized by at most 7 reflections from a finite collection of mirrors.
\end{corollary}

\begin{proof} (Sketch) We use one mirror to transform each normal family to the family of upward oriented vertical rays (as in Levi-Civita's theorem). This reduces the situation to a local diffeomorphism of the horizontal plane. If this diffeomorphism is orientation reversing then we can realize it by four mirror reflection, and the total number of reflections is six.

If the plane diffeomorphism is orientation preserving, we  add one more (penultimate) mirror. The mirror is flat, and it transforms a parallel beam into another parallel beam, reversing the orientation. In this case, the total number of reflections is seven.
\end{proof}

\section{Open questions} \label{problems}

The problem of the minimal number of mirrors needed to realize a diffeomorphism of normal families of rays was not  studied in this paper. In particular, the statement of Corollary \ref{cor3} is hardly optimal.

We finish with several problems.

\begin{question}\label{q1}
What is the minimal number of  mirror reflections needed to realize a diffeomorphism of two bounded domains in $\RRR^2$?
\end{question}

\begin{question}\label{q2}
The same question for local diffeomorphisms of two normal families of oriented lines in $\RRR^3$.
\end{question}

\begin{question}\label{q3}
Generalization of these questions to diffeomorphisms of domains in $\RRR^{n}$ and to diffeomorphisms of normal families in $\RRR^{n+1}$.
\end{question}

\begin{question}\label{q4}
Given a diffeomorphism $f$ of compact domains in $\RRR^n$, what is the least number of gradient diffeomorphisms whose composition is $f$?
\end{question}

\begin{question}\label{q5}
Consider the pencils of lines through two points (perhaps, coinciding) in $\RRR^2$. Which mappings between these pencils can be realized by 2-mirror reflections? Same question in  $\RRR^{n+1}$
\end{question}

\section*{Acknowledgements}

We are grateful to C. Croke, A. Glutsyuk, M. Levi, R. Perline, and B. Sleeman for stimulating discussions.
Example \ref{primer} is due to A. Glutsyuk.

The first  author was supported by Portuguese funds through CIDMA - Center for Research and Development in Mathematics and Applications, and FCT - Portuguese Foundation for Science and Technology, within the project UID/MAT/04106/2013, as well as by the FCT research project PTDC/MAT/113470/2009. The second author was supported by the NSF grants DMS-1105442 and DMS-1510055. The third author was supported by the grant RFBR 15-01-03747

\section*{Appendix}

Here we prove that the reflection in an ellipse, considered as a mapping of the pencil of oriented lines through one focus to the pencil of oriented lines through another focus, is a M\"obius transformation.

Let $A$ and $B$ be the foci, and $C$ be a point of the ellipse. By scaling, assume that $|AC|+|CB|=2$, and let $|AB|=2c$. Set $|AC|=v, |BC|=u$, and denote the angles $BAC$ and $ABC$ by $\alpha$ and $\beta$ (see Fig. \ref{figTr}). Then we have a system of equations
\begin{figure}[h]
\begin{picture}(0,73)
\rput(6.7,0.1){
\scalebox{1.15}{
\pspolygon[linecolor=blue](-2,0)(0.5,2)(2,0)
\rput(-2.25,-0.1){\scalebox{0.87}{$A$}}
\rput(2.18,-0.1){\scalebox{0.87}{$B$}}
\rput(0.65,2.2){\scalebox{0.87}{$C$}}
\rput(-1.02,1.1){\scalebox{0.87}{$v$}}
\rput(1.45,1.1){\scalebox{0.87}{$u$}}
\rput(-1.4,0.22){\scalebox{0.87}{$\al$}}
\rput(1.38,0.27){\scalebox{0.87}{$\bt$}}
}
}
\end{picture}
\caption{Reflection in an ellipse.}
\label{figTr}
\end{figure}
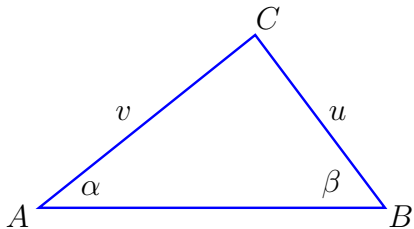
$$
v\sin\alpha=u\sin\beta,\ u+v=2,\ v\cos\alpha+u\cos\beta=2c.
$$
Eliminating $u$ and $v$, we obtain
\begin{equation} \label{sines}
\sin(\alpha+\beta)=c(\sin\alpha+\sin\beta).
\end{equation}

We identify the circles centered at $A$ and $B$ with the projective line via stereographic projections. Let $x=\tan(\alpha/2),\ y=\tan(\beta/2)$ be the respective coordinates in the projective line. Express sines and cosines in (\ref{sines}) in terms of $x$ and $y$:
$$
\frac{2x(1-y^2)}{(1+x^2)(1+y^2)} + \frac{2y(1-x^2)}{(1+x^2)(1+y^2)} = c\left(\frac{2x}{1+x^2} + \frac{2y}{1+y^2}\right)
$$
or, equivalently,
$$
x(1-y^2)+y(1-x^2)=c[x(1+y^2)+y(1+x^2)].
$$
Canceling $x+y$, we find that
$$
y=\left(\frac{1-c}{1+c}\right) \frac{1}{x},
$$
a fractional-linear transformation.

\end{document}